\documentclass{article}

\usepackage{color}
 \catcode`@=11
\usepackage{color}
\usepackage{amsthm,amssymb,amsfonts}
\usepackage{amsmath}
\catcode`@=11 \@addtoreset{equation}{section}

\catcode`@=12

\newtheorem{theorem}{Theorem}[section]
\newtheorem{proposition}{Proposition}[section]
\newtheorem{lemma}{Lemma}[section]
\newtheorem{corollary}{Corollary}[section]
\theoremstyle{definition}

\newtheorem{definition}{Definition}[section]
\newtheorem{remark}{Remark}

\usepackage[colorlinks=true,urlcolor=red,linkcolor=blue,citecolor=red,bookmarks=red]{hyperref}

\title{Boundary controllability for a degenerate and singular wave equation}
\author{{\sc Brahim Allal}
\\
Facult\'e des Sciences et Techniques\\
Universit\'e Hassan 1er\\
Laboratoire MISI, B.P. 577\\
Settat 26000, Morocco\\ 
email: b.allal@uhp.ac.ma
\\
{\sc Alhabib Moumni}\\
Moulay Ismail University of Meknes,\\
FST Errachidia, MAIS Laboratory, MAMCS Group,\\
P.O. Box 509, Boutalamine 52000, Errachidia, Morocco\\
email: alhabibmoumni2020@gmail.com\\
{\sc  Jawad Salhi}
\\
Moulay Ismail University of Meknes,\\
FST Errachidia, MAIS Laboratory, MAMCS Group,\\
P.O. Box 509, Boutalamine 52000, Errachidia, Morocco\\
email: sj.salhi@gmail.com}

\date{}
\begin{document}

\maketitle
\begin{abstract}
In this paper, we deal with the boundary controllability of a one-dimensional degenerate and singular wave equation with degeneracy and singularity occurring at the boundary of the spatial domain. Exact boundary controllability is proved in the range of subcritical/critical potentials and for sufficiently large time, through a boundary controller acting away from the degenerate/singular point. By duality argument, we reduce the problem to an observability estimate for the corresponding adjoint system, which is proved by means of the multiplier method and a new special Hardy-type inequality.
\end{abstract}

\section{Introduction}\label{intro}
Controllability issues for nondegenerate hyperbolic equations have been a mainstream topic over the past several years, and numerous developments have been pursued (see for example \cite{BLR1992,Burq1997,Lions1988b,Russell1967,Zhang2000,Zuazua1993} and the references therein).

In the present paper, we consider the following degenerate/singular hyperbolic equation
\begin{equation}\label{boundarycontrolproblem}
\begin{cases}
y_{tt} - (x^{\alpha}y_x)_x  - \frac{\mu}{x^{2-\alpha}}y =0, & (t,x) \in Q:=(0,T) \times (0,1),\\
y(1)=f,  & t \in (0,T), \\
\left\{
\begin{array}{ll}
y (t, 0) = 0,\quad & \text{if} \;0 \leq \alpha <1,\\
(x^{\alpha} y_{x})(t, 0)= 0, \quad & \text{if} \; 1 \leq \alpha <2,
\end{array}
\right.
& t \in (0,T),
\\
y(0,x)=y_0(x),\quad y_t(0,x)=y_1(x), & x \in (0,1),
\end{cases}
\end{equation}
where $\alpha$ and $\mu$ are two real parameters, $(y, y_t)$ is the state variable, $(y_0, y_1)$ is regarded as being the initial value, $T>0$ stands for the length of the time-horizon and $f\in L^2(0,T)$ is the control at $x = 1$ (that is, away from the degenerate and singular point). 
In particular, if $\alpha \in (0,1)$ we say that the problem is {\it weakly degenerate} (WD), if $\alpha \in [1,2)$ then it is {\it strongly degenerate} (SD).

In order to study problem \eqref{boundarycontrolproblem}, we assume that the parameters $\alpha$ and $\mu$ satisfy the
following assumption:
\begin{equation}\label{basicass}
\alpha \in [0,2)\setminus\{1\}  \text{ and } \mu\leq \mu(\alpha),
\end{equation}
where 
\begin{equation}\label{constantmu}
\mu(\alpha):=\dfrac{(1-\alpha)^2}{4}
\end{equation}
is the constant appearing in the following generalized Hardy inequality: for all $\alpha\in [0, 2)$,
\begin{equation}\label{dshp}
\frac{(1-\alpha)^2}{4} \int_0^1\frac{u^2}{x^{2-\alpha}}\,dx \leq \int_0^1 x^\alpha u_x^2\,dx,
\end{equation}
for all $u \in C^{\infty}_{c}(0, 1)$ (the space of infinitely smooth functions compactly supported
in $(0,1)$). We refer for example to \cite[chap 5.3]{Davies1995}.

We emphasise that \eqref{dshp} ensures that, if $\alpha \in[0,2)\setminus\{1\}$ and if $u \in H_{\text {loc }}^{1}((0,1])$ is such that $x^{\alpha/2} u_{x} \in L^{2}(0,1)$, then $\displaystyle\frac{u}{x^{(2-\alpha)/2}}$ belongs to $L^{2}(0,1)$. On the contrary, in the case $\alpha=1$, \eqref{dshp} (which reduces to a trivial inequality) does not provide this information anymore. Hence, it is not surprizing if with our techniques we cannot handle this latter special case and we refer to \cite{FMpress} and \cite{Van2011} where this issue is attacked in a different way for the heat equation.  

Now, observe that when $\mu=0$, the problem above is purely degenerate. In this case, controllability properties by means of a boundary control have been
investigated in various papers. We refer the reader to the following pioneering contributions  \cite{Alabau2016,BaiandChai,Fardigola2019,Gueye2014,KKL,zhang2017a}.
We also refer to \cite{zhang2017b,zhang2018} for other works on controllability problems by means of a locally distributed control.

On the other hand, when $\alpha = 0$, system \eqref{boundarycontrolproblem} becomes purely singular with a singularity that takes the form of an inverse-square potential. 
As far as we know,  \cite{Cazacu2012} and \cite{VanZua2009} are the unique published works on this subject; they are concerned with the problem of exact controllability
for the linear multidimensional wave equation with singular potentials.

We underline that this is the first paper to consider the exact boundary controllability for the system
\eqref{boundarycontrolproblem} that couples a degenerate variable coefficient in the principal part with a singular potential. More precisely, we will solve the following problem: Given a time-horizon
$T > 0$, initial data $(y_0,y_1)$ and a target $(y_0^T,y_1^T)$, we ask whether there is  a suitable control function $f$ such
that the corresponding solution of the system \eqref{boundarycontrolproblem} satisfies
\begin{equation}\label{ecproperty}
(y,y_t)(T,\cdot)=(y_0^T,y_1^T)(\cdot).
\end{equation}
In the view of the linearity and the time-reversibility of system \eqref{boundarycontrolproblem}, it can be shown that this system is exactly
controllable through the boundary Dirichlet conditions at $x = 1$ if and only if it is
null controllable (see \cite[Proposition 2.3.1]{Zuazua2006}). This guarantees that \eqref{boundarycontrolproblem} is exactly controllable if
and only if, for any $(y_0,y_1)$ there exists a control $f$
such that the corresponding solution $(y,y_t)$ of \eqref{boundarycontrolproblem} satisfies

\begin{equation}\label{nullproperty}
y(T,\cdot)=y_t(T,\cdot)=0.
\end{equation}
This result is actually equivalent to an observability inequality for the solutions
of the adjoint system  
\begin{equation}\label{adjointproblem}
\begin{cases}
u_{tt} - (x^{\alpha}u_x)_x  - \frac{\mu}{x^{2-\alpha}}u =0, & (t,x) \in Q,\\
\left\{
\begin{array}{ll}
u(t, 0) = 0,\quad & \text{if} \;0 \leq \alpha <1,\\
(x^{\alpha} u_{x})(t, 0)= 0, \quad & \text{if} \; 1 < \alpha <2,
\end{array}
\right.
& t \in (0,T),
\\
u(1)=0,  & t \in (0,T), \\
u(0,x)=u_0(x),\quad u_t(0,x)=u_1(x), & x \in (0,1).
\end{cases}
\end{equation}
The main contribution of this paper consists precisely in proving this observability
inequality. To be more precise, we show that, for $T$ large enough and all $\mu\leq \mu(\alpha)$, there
exists $C >0$ such that
\begin{equation}\label{oinequality}
\int_{0}^{1}\left\{u_{t}^{2}(0, x)+x^\alpha u_{x}^{2}(0, x)-\frac{\mu}{x^{2-\alpha}}u^2(0, x)\right\}\,dx \leq C \int_{0}^{T} u_{x}^{2}(t, 1)\,dt.
\end{equation}
We refer to Theorem \ref{observabilityresult} for a precise statement of this result. The proof of \eqref{oinequality} relies on the multiplier method which has first been
given in \cite{Ho1986} and can be found in \cite{Komornik1995,Lions1988,Osses2001}. 
As a consequence of this inequality, it follows that system \eqref{boundarycontrolproblem} is exactly controllable in time $T$ by a
control acting at $x=1$. More precisely, our main result is the boundary exact controllability of \eqref{boundarycontrolproblem} for $T > 0$ sufficiently large and
independent of $\mu\leq \mu(\alpha)$ (see Theorem \ref{controllabilityresult} for a rigorous statement).

Let us notice that in the subcritical case $\mu< \mu(\alpha)$, in view of Lemma \ref{equivalenceresult}, one could proceed as in the
case of the purely degenerate wave equation (when $\mu = 0$). However, this would produce a controllability time $T^{\mu}_{\alpha}$
depending on $\mu$ and such that $T^{\mu}_{\alpha} \rightarrow +\infty$ as $\mu \rightarrow \mu(\alpha)$. Hence the time of controllability $T^{\mu}_{\alpha}$
would not be uniform with respect to the parameter $\mu$. Moreover, by this method, no result could be expected in the critical
case $\mu=\mu(\alpha)$.

In order to overcome this issue, we prove the following new weighted Hardy-type inequality that generalizes \cite[Theorem 1.1]{VanZua2009}, which is crucial in
order to get  a uniform time of controllability and to also treat the critical case $\mu=\mu(\alpha)$.
\begin{theorem}\label{keytool}
Let $\mu(\alpha)$ be as in \eqref{constantmu}. For all $\alpha\in [0, 2)$ and for all $u\in C^{\infty}_{c}(0,1)$, we have
\begin{equation}\label{keyineq}
\int_{0}^{1} x^2 u_x^2\,dx \leq \int_{0}^{1}\left(x^\alpha u_x^2-\mu(\alpha) \frac{u^2}{x^{2-\alpha}} \right)\,dx +\frac{(1-\alpha)(\alpha-3)}{4}\int_{0}^{1} u^2\,dx.
\end{equation}
\end{theorem}
The proof of this theorem is given in an appendix at the end of the paper.
\begin{remark}
Actually, using multiplier method, the authors in \cite{Alabau2016} proved the null controllability of purely
degenerate wave equations with a more general speed coefficient. In the pure power case (i.e. the system \eqref{boundarycontrolproblem} with $\mu=0$), they provided an explicit expression for the controllability time given by
$$
\bar{T}_{\alpha}= \dfrac{1}{2-\alpha}\left(4+2\alpha \min\{2,\dfrac{1}{\sqrt{2-\alpha}}\}\right).
$$
Different from \cite{Alabau2016}, it is worth mentioning that the authors in \cite{zhang2017a}, also proved the null controllability of \eqref{boundarycontrolproblem} when $\mu=0$ by using a spectral approach. In particular, they provided a sharp controllability time given by
$$
T_{\alpha}= \dfrac{4}{2-\alpha}.
$$ 
Here, by means of the multiplier method and the special Hardy-type inequality \eqref{keyineq}, we retrieve the expected minimal time of controllability $T_{\alpha}$, which coincides with the one that the spectral method gives for the purely degenerate
wave equation.
\end{remark}

%%%%%%%%%%%%%%%%%%%%%%%%%%%%%%%%%%%%%%%%%%%%%%%%%%%%%%%%%%%%%%%%%%%%%%%%%%%%%%
\section{Preliminary results}
In this section, we state some lemmas that play an important role in the rest of the paper.
First of all, we prove the following new weighted Hardy-Poincar\'e inequality which is crucial in every situation.
\begin{lemma}\label{keytool0}
Let $\mu(\alpha)$ be as in \eqref{constantmu}. Then, for all $\alpha\in [0, 2)$
\begin{equation}
\label{keyineq0}
\int_{0}^{1} u^2\,dx \leq C_{\alpha} \int_{0}^{1}\left(x^\alpha u_x^2-\mu(\alpha) \frac{u^2}{x^{2-\alpha}} \right)\,dx\qquad \forall u\in C^{\infty}_{c}(0,1),
\end{equation}
where  
\begin{equation*}
%\label{calpha1}
C_{\alpha}:= \frac{16}{(2-\alpha)^2}.
\end{equation*}
\end{lemma}

\begin{proof}
The result will be obtained by the \textit{expansion of the square} method as in \cite{BDE}. Assume that $u\in C^{\infty}_{c}(0,1)$ and let $\lambda> 0$. Then we write
$$
0 \leq \int_{0}^{1}\left(x^{\alpha / 2} u_{x}-\frac{1-\alpha}{2} \frac{1}{x^{(2-\alpha) / 2}} u+\lambda u\right)^{2}\,dx.
$$
Expanding the above inequality, we obtain
$$
\begin{aligned}
0 \leq & \int_{0}^{1} x^{\alpha} u_{x}^{2}\,dx+\frac{(1-\alpha)^{2}}{4} \int_{0}^{1} \frac{1}{x^{2-\alpha}} u^{2}\,dx
+\lambda^2\int_{0}^{1} u^{2}\,dx-\frac{1-\alpha}{2} \int_{0}^{1} \frac{1}{x^{1-\alpha}}\left(u^{2}\right)_{x}\,dx\\
&+\lambda\int_{0}^{1} x^{\alpha / 2}\left(u^{2}\right)_{x}\,dx-\lambda (1-\alpha) \int_{0}^{1} \frac{1}{x^{(2-\alpha) / 2}} u^{2}\,dx.
\end{aligned}
$$
Then integrations by parts lead to
$$
\begin{aligned}
0 &\leq  \int_{0}^{1} x^{\alpha} u_{x}^{2}\,dx+\frac{(1-\alpha)^{2}}{4} \int_{0}^{1} \frac{1}{x^{2-\alpha}} u^{2}\,dx+\lambda^2\int_{0}^{1} u^{2}\,dx \\
&\quad -\frac{(1-\alpha)^{2}}{2} \int_{0}^{1} \frac{1}{x^{2-\alpha}} u^{2}\,dx-\lambda (1-\frac{\alpha}{2}) \int_{0}^{1} x^{(\alpha-2) / 2} u^{2}\,dx\\
&= \int_{0}^{1}\left(x^{\alpha} u_{x}^{2}-\mu(\alpha)\frac{u^2}{x^{2-\alpha}}\right)\,dx + \int_{0}^{1} \left[\lambda^2 -\lambda (1-\frac{\alpha}{2}) x^{(\alpha-2) / 2}\right] u^{2}\,dx.
\end{aligned}
$$ 
Using the fact that $\alpha\in[0,2)$, we deduce that
$$
0 \leq \int_{0}^{1}\left(x^{\alpha} u_{x}^{2}-\mu(\alpha)\frac{u^2}{x^{2-\alpha}}\right)\,dx + \left[\lambda^2 -\lambda (1-\frac{\alpha}{2})\right] \int_{0}^{1}  u^{2}\,dx.
$$
Now, observe that 
$$
\lambda^2 -\lambda (1-\frac{\alpha}{2}) \geq - \frac{(2-\alpha)^2}{16},
$$ 
with equality if and only if $\lambda = \frac{2-\alpha}{4}$. By choosing this value, we get the result.
\end{proof}
By the definition of $\alpha$ and \eqref{keyineq} combined with \eqref{keyineq0}, we obtain the following inequalities which are needed for proving a refined trace result (see Lemma \ref{hiddenregularity}).
\begin{lemma}\label{keytool1}
Let $\mu(\alpha)$ be as in \eqref{constantmu}.
\begin{enumerate}
\item If $\alpha\in [0, 1)$, then 
\begin{equation}\label{keyineq1}
\int_{0}^{1} x^2 u_x^2\,dx 
\leq \int_{0}^{1}\left(x^\alpha u_x^2-\mu(\alpha) \frac{u^2}{x^{2-\alpha}} \right)\,dx\qquad \forall u\in C^{\infty}_{c}(0,1).
\end{equation}
\item If $\alpha\in [1, 2)$, then
\begin{equation}
\label{keyineq2}
\int_{0}^{1} x^2 u_x^2\,dx \leq C^{\prime}_{\alpha}\int_{0}^{1}\left(x^\alpha u_x^2-\mu(\alpha) \frac{u^2}{x^{2-\alpha}} \right)\,dx\qquad \forall u\in C^{\infty}_{c}(0,1),
\end{equation}
where  
\begin{equation*}
%\label{calpha}
C^{\prime}_{\alpha}:=\left[1 + \frac{4(1-\alpha)(\alpha-3)}{(2-\alpha)^2}\right].
\end{equation*}
\end{enumerate}
\end{lemma}
\begin{remark}
In the weakly degenerate case, observe that owing to \eqref{keyineq} and \eqref{keyineq0}, we obtain two different bounds for $\|u\|^{2}_{L^2(0,1)}$ in terms of $\int_{0}^{1}\left(x^\alpha u_x^2-\mu(\alpha) \frac{u^2}{x^{2-\alpha}} \right)\,dx$. By taking the minimum of the two
corresponding constants, one can deduce the following sharp Hardy-Poincar\'e inequality: for all $\alpha\in [0, 1)$
\begin{equation*}
%\label{keyineq00}
\int_{0}^{1} u^2\,dx \leq c_{\alpha} \int_{0}^{1}\left(x^\alpha u_x^2-\mu(\alpha) \frac{u^2}{x^{2-\alpha}} \right)\,dx\qquad \forall u\in C^{\infty}_{c}(0,1),
\end{equation*}
where  
\begin{equation*}
c_{\alpha}:= \min\left(\frac{4}{(1-\alpha)(3-\alpha)},\frac{16}{(2-\alpha)^2}\right).
\end{equation*}
\end{remark}

%%%%%%%%%%%%%%%%%%%%%%%%%%%%%%%%%%%%%%%%%%%%%%%%%%%%%%%%%%%%%%%%%%%%%%%%%%%%%%%%%%%%%%%%%%%%%%%%%%%%%%%%%%%%%%%%%%%%%%%%%%%%%%%%%%%%%%%%%
%%%%%%%%%%%%%%%%%%%%%%%%%%%%%%%%%%%%%%%%%%%%%%%%%%%%%%%%%%%%%%%%%%%%%%%%%%%%%%%%%%%%%%%%%%%%%%%%%%%%%%%%%%%%%%%%%%%%%%%%%%%%%%%%
\section{Basic properties for the degenerate and singular wave equation}\label{section2}
Before considering controllability issues, we address the question of well-posedness of the degenerate and singular hyperbolic problem \eqref{boundarycontrolproblem}.
To this end, we first need to state some basic properties of nonhomogeneous wave equations of the following type:
\begin{equation}\label{problem}
\begin{cases}
u_{tt} - (x^{\alpha}u_x)_x  - \frac{\mu}{x^{2-\alpha}}u =h, & (t,x) \in Q,\\
\left\{
\begin{array}{ll}
u(t, 0) = 0,\quad & \text{if} \;0 \leq \alpha <1,\\
(x^{\alpha} u_{x})(t, 0)= 0, \quad & \text{if} \; 1 < \alpha <2,
\end{array}
\right.
& t \in (0,T),
\\
u(1)=0,  & t \in (0,T), \\
u(0,x)=u_0(x),\quad u_t(0,x)=u_1(x), & x \in (0,1),
\end{cases}
\end{equation}
where $h\in L^2(Q)$ is a given source term.

\subsection{Finite energy solutions for homogeneous Dirichlet boundary conditions}
The first step is to prove the existence of finite energy solutions for \eqref{problem}.
In this purpose, we briefly recall some usual weighted Sobolev spaces that
are naturally associated with degenerate problems (see \cite{CMV2008}).
For all $0 \leq \alpha<2$, we consider the following weighted Hilbert space
$$
H^{1}_{\alpha}(0,1):=\left\{u \in L^{2}(0,1) \cap H_{l o c}^{1}((0,1]) \mid x^{\alpha / 2} u_{x} \in L^{2}(0,1)\right\},
$$
with the norm
$$
\|u\|^{2}_{H^{1}_{\alpha}(0,1)}:= \|u\|^{2}_{L^{2}(0,1)} + \|x^{\alpha/2}u_x\|^{2}_{L^{2}(0,1)}.
$$
Obviously, for any $u \in H^{1}_{\alpha}(0,1)$, the trace at $x=1$ exists. On the other hand, the trace
of $u$ at $x = 0$ only makes sense when $0\leq \alpha <1$ (see for example \cite{CRV}). For this reason, we consider the following space $H^{1}_{\alpha,0}$ depending on the value of $\alpha$: 
\begin{description}
\item[(i)] For $0 \leq \alpha<1$, we define
$$
H_{\alpha, 0}^{1}(0,1):=\left\{u \in H_{\alpha}^{1}(0,1) \mid u(0)=u(1)=0\right\}.
$$
\item[(ii)] For $1 \leq \alpha<2$, we define
$$
H_{\alpha, 0}^{1}(0,1):=\left\{u \in H_{\alpha}^{1}(0,1) \mid u(1)=0\right\}.
$$
\end{description}
We recall that in both cases, $H_{\alpha, 0}^{1}(0,1)$ is the closure of $C^{\infty}_{c}(0,1)$ for the
norm $\|\cdot\|_{H_{\alpha}^{1}(0,1)}$, see for example \cite{CRV}. Therefore one can deduce that \eqref{dshp} holds true for any $u\in H_{\alpha, 0}^{1}(0,1)$. Then, one can show that
$$
\forall u\in H_{\alpha,0}^{1}(0,1),\qquad \|u\|_{H^{1}_{\alpha,0}(0,1)}:= \left(\int_{0}^{1}x^{\alpha}u_x^2\,dx\right)^{\frac{1}{2}},
$$
defines a norm on $ H_{\alpha,0}^{1}(0,1)$ that is equivalent to $\|\cdot\|^{2}_{H^{1}_{\alpha}(0,1)}$.

Next, we also set
$$
H^{2}_{\alpha}(0,1):=\left\{u \in H^{1}_{\alpha,0}(0,1)\mid x^{\alpha}u_x \in H^1(0,1)\right\},
$$
where $H^1(0,1)$ denotes the classical Sobolev space of all functions $u\in  L^2(0,1)$ such
that $u_x \in L^2(0,1)$.

Let us pass to introduce the functional setting associated to the degenerate/singular problems (see \cite{Biccari2021} or \cite{Van2011}). For any $\mu \leq \mu(\alpha)$, we consider the Hilbert space $H_{\alpha}^{1,\mu}(0,1)$ given by 
$$
H_{\alpha}^{1,\mu}(0,1):=\left\{u \in L^{2}(0,1) \cap H_{l o c}^{1}((0,1]) \mid \int_{0}^{1}\left(x^{\alpha} u_{x}^{2}-\frac{\mu}{x^{2-\alpha}} u^{2}\right)\,dx<+\infty\right\}
$$
endowed with the scalar product
$$
\langle u,v \rangle_{H_{\alpha}^{1,\mu}}:= \int_{0}^{1} uv+   x^{\alpha} u_xv_x - \frac{\mu}{x^{2-\alpha}}uv\,dx.
$$
According to \cite{Van2011}, the trace at $x = 0$ of any $u\in H_{\alpha}^{1,\mu}(0,1)$ makes
sense as soon as $\alpha < 1$. This leads us to introduce the next space:
\begin{description}
\item[(i)] For $0 \leq \alpha<1$, we define
$$
H_{\alpha,0}^{1,\mu}(0,1):=\left\{u \in H^{1,\mu}_{\alpha}(0,1) \mid u(0)=u(1)=0\right\}.
$$
\item[(ii)] For $1 < \alpha<2$, we change the definition of $H_{\alpha,0}^{1,\mu}(0,1)$ in the following way
$$
H_{\alpha,0}^{1,\mu}(0,1):=\left\{u \in H^{1,\mu}_{\alpha}(0,1) \mid u(1)=0\right\}.
$$
\end{description}
Let us mention that in both cases, $H_{\alpha,0}^{1,\mu}(0,1)$ may be seen as the closure of $C^{\infty}_{c}(0,1)$ with respect to the norm induced by $\langle \cdot,\cdot \rangle_{H_{\alpha}^{1,\mu}}$ and thus \eqref{dshp}, \eqref{keyineq} and \eqref{keyineq0} also hold true in $H_{\alpha,0}^{1,\mu}(0,1)$. Therefore, thanks to \eqref{dshp}, one can see that
$$
\|u\|_{H_{\alpha,0}^{1,\mu}(0,1)}:=\left(\int_{0}^{1}x^{\alpha} u_x^{2}- \frac{\mu}{x^{2-\alpha}}u^2\,dx\right)^{\frac{1}{2}}
$$
defines a norm on $H_{\alpha,0}^{1,\mu}(0,1)$ which is equivalent to $\|\cdot\|_{H_{\alpha}^{1,\mu}(0,1)}$. Hence $H_{\alpha,0}^{1,\mu}(0,1)$ is a Hilbert space for the scalar product
$$
\langle u,v \rangle_{H_{\alpha,0}^{1,\mu}}:= \int_{0}^{1}  x^{\alpha} u_xv_x - \frac{\mu}{x^{2-\alpha}}uv\,dx.
$$
Moreover, we also remark that in the case of a sub-critical parameter $\mu<\mu(\alpha)$, thanks to \eqref{dshp}, it is easy to see that $H^{1,\mu}_{\alpha, 0}(0,1)=H^{1}_{\alpha, 0}(0,1)$. On the contrary, for the critical value $\mu=\mu(\alpha)$, the space is enlarged (see \cite{VazZua2000} for this observation in the case $\alpha=0$):
$$
H^{1}_{\alpha, 0}(0,1) \subsetneq H^{1,\mu(\alpha)}_{\alpha, 0}(0,1).
$$
To be more precise, in the subcritical case, one can prove the next result.
\begin{lemma}\label{equivalenceresult}
Assume that $\alpha\in [0,2)\setminus\{1\}$ and $\mu<\mu(\alpha)$. Then there exist two constants $C^{1}_{\alpha,\mu}>0$ and $C^{2}_{\alpha,\mu}>0$ such that, for every
$u\in H^{1}_{\alpha, 0}(0,1)$
$$
C^{1}_{\alpha,\mu}\|u\|^{2}_{H^{1}_{\alpha, 0}(0,1)}\leq  \|u\|^{2}_{H^{1,\mu}_{\alpha, 0}(0,1)} \leq C^{2}_{\alpha,\mu}\|u\|^{2}_{H^{1}_{\alpha, 0}(0,1)}.
$$
More precisely,
$$
C^{1}_{\alpha,\mu}= 1-\frac{\max(0,\mu)}{\mu(\alpha)},\quad C^{2}_{\alpha,\mu}= 1-\frac{\min(0,\mu)}{\mu(\alpha)}.
$$
\end{lemma}

Further, we define $H^{-1,\mu}_{\alpha,0}(0,1)$  the dual space of $H_{\alpha,0}^{1,\mu}(0,1)$ with respect to the pivot space
$L^2(0,1)$, endowed with the natural norm
\begin{equation*}
\| f \|_{H_{\alpha,0}^{-1,\mu}(0,1)}:=  \sup_{\| g \|_{H_{\alpha,0}^{1,\mu}(0,1)}=1} \langle f, g \rangle_{H^{-1,\mu}_{\alpha,0}(0,1), H_{\alpha,0}^{1,\mu}(0,1)}.
\end{equation*}
Besides, we also set
$$
H_{\alpha}^{2,\mu}(0,1):=\left\{u \in H_{\alpha}^{1,\mu}(0,1) \cap H_{l o c}^{2}((0,1]) \mid\left(x^{\alpha} u_{x}\right)_{x}+\frac{\mu}{x^{2-\alpha}} u \in L^{2}(0,1)\right\}.
$$
Finally, for all $\mu \leq \mu(\alpha)$, we define the operator
$$A^{\mu}_\alpha u:= (x^\alpha u_x)_x + \dfrac{\mu}{x^{2-\alpha}}u$$
with domain
$$
D^{\mu}_{\alpha}:=D(A^{\mu}_{\alpha})= H^{1,\mu}_{\alpha,0}(0,1)\cap H_{\alpha}^{2,\mu}(0,1).
$$
Now, we are ready for the well posedness result of \eqref{problem}.
\begin{theorem}\label{wellposed}
Assume that \eqref{basicass} holds.
\begin{description}
\item[(i)] For every $(u_0, u_1)\in H_{\alpha,0}^{1,\mu}(0,1)\times L^2(0,1)$ and $h\in L^2(Q)$, there exists a unique mild (or weak)
solution $u \in C([0,T]; H_{\alpha,0}^{1,\mu}(0,1)) \cap C^1([0,T];L^{2}(0,1))$ of \eqref{problem} satisfying the following estimate:
\begin{equation*}
\|(u(t),u_t(t))\|_{H_{\alpha,0}^{1,\mu}\times L^2} \leq C\Big( \|(u_0,u_1)\|_{H_{\alpha,0}^{1,\mu}\times L^2} +\|h\|_{L^1((0,T);L^2(0,1))}\Big),\quad \forall t\in [0,T].
\end{equation*}
\item[(ii)] Moreover, if $\left(u_{0}, u_{1}\right) \in D^{\mu}_{\alpha} \times H_{\alpha,0}^{1,\mu}(0,1)$ and $h \in C^{1}\left([0, T] ; L^{2}(0,1)\right)$, then problem \eqref{problem} admits a unique classical solution $
u \in C\left([0, T] ; D^{\mu}_{\alpha}\right) \cap C^{1}([0,T]; H_{\alpha,0}^{1,\mu}(0,1)) \cap C^{2}\left([0, T]; L^{2}(0,1)\right)$ that satisfies
\begin{equation*}
\|(u(t),u_t(t))\|_{D_{\alpha}^{\mu}\times L^2} \leq C\Big(\|(u_0,u_1)\|_{D_{\alpha}^{\mu}\times L^2} +\|h\|_{L^1((0,T);L^2(0,1))}\Big),\quad \forall t\in [0,T].
\end{equation*}
\end{description}
\end{theorem}

\begin{proof}
Observe that the evolution equation \eqref{problem} is equivalent to
\begin{equation}\label{problem1}
\begin{cases}
u_t(t)=v(t), & v_t(t) = A_{\alpha}^{\mu} u(t) + h(t),\\
u(0)=u_0, & v(0)=u_1.
\end{cases}
\end{equation}
Now, we introduce the energy space
$$
\mathcal{H}:= H_{\alpha,0}^{1,\mu}(0,1)\times L^2(0,1),
$$
equipped with the inner product defined by
$$
\langle U_1,U_2 \rangle_{\mathcal{H}}:= \int_{0}^{1} x^{\alpha} u_{1,x} u_{2,x} - \frac{\mu}{x^{2-\alpha}}u_1 u_2 + v_1 v_2\,dx,
$$
where
$$
U_i=(u_i,v_i)^{T}\in \mathcal{H},\quad i=1,2.
$$
We know that $\mathcal{H}$ is a Hilbert space equipped with the adequate scalar product $\langle \cdot,\cdot \rangle_{\mathcal{H}}$.

The differential system \eqref{problem1} can be rewritten in the abstract form
$$
(CP)\quad
\begin{cases} \frac{dU}{dt}(t) = \mathcal{A}^{\mu}_{\alpha}U(t) + F(t),\\
U(0)=\left(\begin{smallmatrix}
u_0\\
u_1
\end{smallmatrix}\right),
\end{cases}
$$
where
$$U(t)=\left(\begin{matrix}
u(t)\\
\frac{du}{dt}(t)
\end{matrix}
\right) \quad \text{and} \quad
F(t)=\left(\begin{matrix}
0\\
h(t)
\end{matrix}
\right).
$$
Here $\mathcal{A}^{\mu}_{\alpha}$ is the unbounded linear operator defined by
$\mathcal{A}^{\mu}_{\alpha}: D(\mathcal{A}^{\mu}_{\alpha})\subset \mathcal{H} \rightarrow \mathcal{H}$
\begin{align*}
&\mathcal{A}^{\mu}_{\alpha} \left(\begin{matrix}u\\v \end{matrix}\right):= \left(\begin{matrix} v\\ A_{\alpha}^{\mu} u \end{matrix}\right),\quad \forall \left(\begin{matrix}u\\v \end{matrix}\right)\in D(\mathcal{A}^{\mu}_{\alpha}),
\end{align*}
with domain
$$
D(\mathcal{A}^{\mu}_{\alpha})=\{(u,v) \in \mathcal{H}\mid  u \in D(A_{\alpha}^{\mu}), v\in H_{\alpha,0}^{1,\mu}(0,1) \}.
$$
We claim that $\mathcal{A}^{\mu}_{\alpha}$ is maximal dissipative on $\mathcal{H}$. Indeed, for $U=(u,v)^{T}\in D(\mathcal{A}^{\mu}_{\alpha})$, we have
\begin{align*}
\langle \mathcal{A}^{\mu}_{\alpha} U, U\rangle_{\mathcal{H}} &= \langle \left(\begin{matrix}v\\A_{\alpha}^{\mu} u \end{matrix}\right),
\left(\begin{matrix}u\\v\end{matrix}\right)\rangle_{\mathcal{H}}\\
&= \int_{0}^{1} x^{\alpha} u_{x} v_{x} - \frac{\mu}{x^{2-\alpha}}u v\,dx + \int_{0}^{1}\Big( (x^{\alpha}u_x)_x + \frac{\mu}{x^{2-\alpha}}u \Big) v\,dx\\
&= 0.
\end{align*}
Hence $\mathcal{A}^{\mu}_{\alpha}$ is dissipative.

Next, we are going to show that $I - \mathcal{A}^{\mu}_{\alpha}$ is surjective. Given a vector $g=(g_1,g_2)\in \mathcal{H}$, we seek a solution $(u,v)\in D(\mathcal{A}^{\mu}_{\alpha})$ of
\begin{equation*}
(I - \mathcal{A}^{\mu}_{\alpha})(u,v)^{T} = (g_1,g_2)^{T}.
\end{equation*}
This is equivalent to the following system:
\begin{align}
&u-A_{\alpha}^{\mu}u = g_1 + g_2, \label{equ1}
\\
&v = u - g_1. \label{equ2}
\end{align}
Suppose that we have found $u$ with the appropriate regularity, then we solve \eqref{equ2} and we obtain $v\in H_{\alpha,0}^{1,\mu}(0,1)$.

Now we state the process on how to get $u$. For this, we recall that $H^{1,\mu}_{\alpha,0}(0,1)$ is a Hilbert space for the scalar product $\langle \cdot,\cdot \rangle_{H^{1,\mu}_{\alpha}}$.
Consequently, for all $(g_1,g_2)^{T}\in \mathcal{H}$, there exists a unique $u \in H^{1,\mu}_{\alpha,0}(0,1)$
such that
$$
\forall \psi \in H^{1,\mu}_{\alpha,0}(0,1),\quad \langle u,\psi \rangle_{H^{1,\mu}_{\alpha}}= \int_{0}^{1} (g_1+g_2)\psi\,dx.
$$
Since $C^{\infty}_{c}(0,1)\subset H^{1,\mu}_{\alpha,0}(0,1)$, we have
$$
\int_{0}^{1} u\psi+   x^{\alpha} u_x\psi_x - \frac{\mu}{x^{2-\alpha}}u\psi\,dx= \int_{0}^{1} (g_1+g_2)\psi\,dx\qquad \forall \psi \in C^{\infty}_{c}(0,1).
$$
By duality, this implies that
$$
u - (x^{\alpha}u_x)_x - \frac{\mu}{x^{2-\alpha}}u =g_1+g_2,
$$
in the sense of distribution. Hence $u\in D(A_{\alpha}^{\mu})$ and 
$$
u- A_{\alpha}^{\mu} u =g_1+g_2 \quad \text{ a.e. in } (0,1),
$$
and the claim follows.
%By the proof of \cite[Proposition 3]{Van2011} (see also \cite{FMpress}), we know that the problem \eqref{equ1}  has a unique weak solution $u\in D(A_{\alpha}^{\mu})$. Next, %we solve \eqref{equ2} and we obtain $v\in H_{\alpha,0}^{1,\mu}(0,1)$. Therefore $(u,v)\in D(\mathcal{A}^{\mu}_{\alpha})$
%and $(I - \mathcal{A}^{\mu}_{\alpha})(u,v) = (g_1, g_2)$, so that $\mathcal{A}^{\mu}_{\alpha}$ is m-dissipative.

Applying the Hille-Yosida theorem (see \cite[Theorem 4.5.1]{barbu} or  \cite[Theorem A.7]{Coron2007}),
we conclude that if $F\in  C^1([0, T]; \mathcal{H})$, i.e., $h\in C^1([0, T]; L^2(0,1))$ and $(u_0,u_1)\in D(\mathcal{A}^{\mu}_{\alpha})$, i.e., $u_0\in H^{1,\mu}_{\alpha,0}(0,1), A_{\alpha}^{\mu} u_0 \in L^2(0,1)$, $u_1 \in H^{1,\mu}_{\alpha,0}(0,1)$, then system \eqref{problem1} (equivalently problem \eqref{problem}) has a
unique solution
\begin{align*}
&u \in  C^1([0,T]; H^{1,\mu}_{\alpha,0}(0,1)),\qquad \frac{du}{dt}\in  C^1([0,T]; L^2(0,1)),\\
&A_{\alpha}^{\mu} u \in C([0,T]; L^2(0,1)).
\end{align*}
If $u_0\in H^{1,\mu}_{\alpha,0}(0,1)$, $u_1\in L^2(0,1)$ and $h\in L^2(Q)$ then the Cauchy
problem (CP) has a mild solution
$$
(u,u_t)\in C([0,T]; H^{1,\mu}_{\alpha,0}(0,1))\times C([0, T]; L^2(0,1)).
$$
\end{proof}
\begin{remark}\label{remark2}
In order to prove our main controllability result we use the so-called multiplier method. The justification of the computations may sometimes be
delicate since we work in non standard weighted spaces, specially in the critical case.
For this reason, we make formal computations that may be justified by the regularization
process described in \cite[Remark 5]{Van2011}. More precisely, one can define a regularized operator 
$$A^{\mu,n}_{\alpha}u:= (x^\alpha u_x)_x + \dfrac{\mu}{(x+\frac{1}{n})^{2-\alpha}}u$$
whose domain is the same as in the purely degenerate case, that is to say $D_{\alpha}^{0}:=H^{2}_{\alpha}(0,1)$ (see \cite{CMV2008}).
Therefore the classical solutions $u^n$ of the regularized problem possess all the regularity required to justify the computations (see \cite{Alabau2016}).
Passing to the limit as $n \rightarrow +\infty$, we recover the result for the solutions $u$ of \eqref{problem}. In what follows, we directly write the computations formally for the solutions $u$ of \eqref{problem}.
\end{remark}
For all $\mu\leq \mu(\alpha)$, we define the generalized energy of a mild solution $u$ of \eqref{adjointproblem} by:
\begin{equation}
\label{energy}
\begin{aligned}
E_{u}(t)&=\frac{1}{2} \int_{0}^{1}\left\{u_{t}^{2}(t, x)+x^\alpha u_{x}^{2}(t, x)-\frac{\mu}{x^{2-\alpha}}u^2(t, x)\right\}\,dx\\
&=\frac{1}{2} \left[\|u_t(t)\|_{L^2(0,1)}^2 + \|u(t)\|_{H^{1,\mu}_{\alpha,0}}^2\right], \quad \forall t \geq 0.
\end{aligned}
\end{equation}
Classical computations show that the generalized energy $E_{u}$ of the solution is constant.
\begin{proposition}
Assume \eqref{basicass} holds and consider $(u_0,u_1)\in H^{1,\mu}_{\alpha,0}(0,1)\times L^2(0,1)$. Then the energy $t \mapsto E_{u}(t)$ of the mild solution $u$ of \eqref{adjointproblem} is constant in time, i.e.
\begin{equation}
\label{constenergy}
E_{u}(t)=E_{u}(0), \quad \forall t \geq 0.
\end{equation}
\end{proposition}
\begin{proof}
Suppose first that $u$ is a classical solution of \eqref{adjointproblem}. Then, multiplying the
equation by $u_t$ and integrating by parts, we obtain
\begin{equation*}
\begin{aligned}
0 &=\int_{0}^{1} u_{t}(t, x)\left\{u_{t t}(t, x)-\left(x^\alpha u_{x}(t, x)\right)_{x}-\frac{\mu}{x^{2-\alpha}}u(t,x)\right\}\,dx \\
&=\overbrace{\int_{0}^{1}\left\{u_{t}(t, x) u_{t t}(t, x)+x^\alpha u_{x}(t, x) u_{t x}(t, x)-\frac{\mu}{x^{2-\alpha}} u_{t}(t, x)u(t,x)\right\}\,dx}^{=\frac{d}{d t} E_{u}(t)}\\
&\quad-\left[x^\alpha u_{t}(t, x) u_{x}(t, x)\right]_{x=0}^{x=1}.
\end{aligned}
\end{equation*}
Using the boundary conditions, we see that the boundary terms vanish. We conclude that the energy of $u$ is constant.
The same conclusion can be extended to any mild solution by an approximation
argument.
\end{proof}
We give now some important lemmas that we shall need to handle boundary conditions in the proof of the multiplier identity \eqref{egaconti}.
\begin{lemma}\label{lemma0}
Let $0\leq \alpha <1$ be given. Then, for all $u \in H_{\alpha}^{2}(0,1)$,
\begin{equation}\label{bt1}
x^{\alpha-1} u^{2}(x) \rightarrow 0 \quad \text { as } x \rightarrow 0^{+}.
\end{equation}
\end{lemma}
\begin{proof}
By the definition of $H_{\alpha}^{2}(0,1)$ in the case $0\leq \alpha <1$, we know that $x^{\alpha} u_{x} \in H^1(0,1)\subset L^{\infty}(0,1)$ and $u(0)=0$. Then,
$$
\left|u_{x}(x)\right| \leq c x^{-\alpha} \text { and } |u(x)|=\left|\int_{0}^{x} u_{x}(\sigma)\,d \sigma\right| \leq c x^{1-\alpha}.
$$
This implies $\left|x^{\alpha-1} u^2(x)\right| \leq c x^{1-\alpha}$. Thus, $\alpha<1$ yields the claim.
\end{proof}

\begin{lemma}\label{lemma1}
Let $1<\alpha <2$ be given. Then, for all $u \in H_{\alpha,0}^{1}(0,1)$,
\begin{equation}\label{bt1}
x^{\alpha-1} u^{2}(x) \rightarrow 0 \quad \text { as } x \rightarrow 0^{+}.
\end{equation}
\end{lemma}
\begin{proof}
Let $u$ be given in $H_{\alpha,0}^{1}(0,1)$. By the definition of $H_{\alpha,0}^{1}(0,1)$, we know that $u \in L^{2}(0,1)$ and $x^{\alpha / 2} u_{x} \in L^{2}(0,1)$. Then
$x^{\alpha-1} u^{2} \in L^{1}(0,1)$. Moreover, we have:
$$
\left(x^{\alpha-1} u^{2}\right)_{x}=(\alpha-1)x^{\alpha-2}u^{2}+2 x^{\alpha-1} u u_{x}.
$$
By \eqref{dshp}, it is easy to see that  $x^{\alpha-2}u^{2} \in L^{1}(0,1)$  and  $x^{\alpha-1} u u_{x}=\left(x^{\alpha / 2 -1} u\right)\left(x^{\alpha / 2} u_{x}\right) \in L^{1}(0,1)$. Hence, $x^{\alpha-1} u^{2} \in W^{1,1}(0,1)$. Thus, $x^{\alpha-1} u^{2} \rightarrow L \geq 0$ as $x \rightarrow 0^{+}$. Finally, $L=0$ since $L \neq 0$ would imply $x^{\alpha / 2 -1} u \notin L^{2}(0,1)$. This completes the proof.
\end{proof}

\subsection{Hidden regularity result}

In this subsection, we prove a regularity result for \eqref{adjointproblem} which is often called hidden
regularity result.

\begin{lemma}\label{hiddenregularity}
Assume that \eqref{basicass} holds. For any mild solution $u$ of \eqref{adjointproblem} we have that $u_{x}(., 1) \in L^{2}(0, T)$ for
every $T> 0$ and
\begin{equation}\label{directineq}
\begin{aligned}
&\int_{0}^{T} u_{x}^{2}(t, 1)\,dt \leq\left(2T + 4\right) E_{u}(0),\qquad \text{if}\quad 0\leq \alpha <1 \quad \text{and}\\
&\int_{0}^{T} u_{x}^{2}(t, 1)\,dt \leq\left(2T + 4C^{\prime}_{\alpha}\right) E_{u}(0),\qquad \text{if}\quad 1\leq \alpha<2,
\end{aligned}
\end{equation}
where $C^{\prime}_{\alpha}$ is the constant appearing in \eqref{keyineq2}.

Moreover,
\begin{equation}
\label{egaconti}
\begin{aligned}
\int_{0}^{T} u_{x}^{2}(t, 1)\,dt &=\int_{0}^{T}\!\!\!\int_{0}^{1}\left\{u_{t}^{2}(t, x)+\left(1-\alpha\right)\left(x^\alpha u_{x}^{2}(t, x)-\frac{\mu}{x^{2-\alpha}}u^2(t,x)\right)\right\}\,dx\,dt \\
&+2\left[\int_{0}^{1} x u_{x}(t, x) u_{t}(t, x)\,dx\right]_{t=0}^{t=T}.
\end{aligned}
\end{equation}
\end{lemma}
\begin{proof} The proof relies on the multiplier method combined with the new Hardy-type inequality \eqref{keyineq}.
Suppose first that $\left(u_{0}, u_{1}\right) \in D^{\mu}_{\alpha} \times H_{\alpha,0}^{1,\mu}(0,1)$, so that $u$ is a classical
solution of \eqref{adjointproblem}. Then, by multiplying \eqref{adjointproblem} by $x u_{x}$ and integrating over $(0, T)\times(0,1)$,
we obtain
\begin{equation}
\label{prbxux}
\begin{aligned}
0 &=\int_{0}^{T}\!\!\!\int_{0}^{1} x u_{x}(t, x)\left(u_{t t}(t, x)-\left(x^\alpha u_{x}(t, x)\right)_{x}-\frac{\mu}{x^{2-\alpha}}u(t,x)\right)\,dx\,dt \\
&=\left[\int_{0}^{1} x u_{x}(t, x) u_{t}(t, x)\,dx\right]_{t=0}^{t=T}-\int_{0}^{T}\!\!\!\int_{0}^{1} x u_{t x}(t, x) u_{t}(t, x)\,dx\,dt \\
&\quad -\int_{0}^{T}\!\!\!\int_{0}^{1}\left(\alpha x^\alpha u_{x}^{2}(t, x)+x^{\alpha+1}  u_{x}(t, x) u_{x x}(t, x)\right)\,dx\,dt-\mu\int_{0}^{T}\!\!\!\int_{0}^{1} x^{\alpha-1}u_x(t,x)u(t,x)\,dx\,dt \\
&=\left[\int_{0}^{1} x u_{x}(t, x) u_{t}(t, x)\,dx\right]_{t=0}^{t=T}-\int_{0}^{T}\!\!\!\int_{0}^{1}\alpha x^\alpha u_{x}^{2}(t, x)\,dx\,dt \\
&\quad-\int_{0}^{T}\!\!\!\int_{0}^{1}\left( x\left(\frac{u_{t}^{2}(t, x)}{2}\right)_{x}+x^{\alpha+1}\left(\frac{u_{x}^{2}(t, x)}{2}\right)_{x}+\mu x^{\alpha-1}\left(\frac{u^2(t,x)}{2}\right)_{x}\right)\,dx\,dt.
\end{aligned}
\end{equation}
Arguing as in the proof of \cite[Lemma 3.2]{Alabau2016}, integrations by parts lead to
\begin{equation}\label{therm1}
\begin{aligned}
&\int_{0}^{T}\!\!\!\int_{0}^{1} x\left(\frac{u_{t}^{2}(t, x)}{2}\right)_{x}\,dx\,dt=-\frac{1}{2} \int_{0}^{T}\!\!\!\int_{0}^{1}u_{t}^{2}(t, x)\,dx\,dt \text{ and}\\
&\int_{0}^{T}\!\!\!\int_{0}^{1} x^{\alpha+1}\left(\frac{u_{x}^{2}(t, x)}{2}\right)_{x}\,dx\,dt= \frac{1}{2} \int_{0}^{T} u_{x}^{2}(t, 1)\,dt-\frac{(\alpha+1)}{2} \int_{0}^{T}\!\!\!\int_{0}^{1} x^\alpha u_{x}^{2}(t, x)\,dx\,dt.
\end{aligned}
\end{equation}
We proceed to integrate by parts the last term in the right hand side of \eqref{prbxux}. We obtain
\begin{equation*}
\int_{0}^{T}\!\!\!\int_{0}^{1}x^{\alpha-1}\left(\frac{u^2(t,x)}{2}\right)_{x}\,dx\,dt
=\frac{1}{2} \int_{0}^{T}\left[x^{\alpha-1} u^{2}(t, x)\right]_{x=0}^{x=1}\,dt -\frac{(\alpha-1)}{2} \int_{0}^{T}\!\!\!\int_{0}^{1} \frac{u^{2}(t, x)}{x^{2-\alpha}}\,dx\,dt.
\end{equation*}
Clearly $\left.x^{\alpha-1} u^2 \right|_{x=1}=0$. Now, we show also that
\begin{equation}\label{keystima}
\left.x^{\alpha-1} u^2 \right|_{x=0}=0.
\end{equation}
To this end, we recall that according to Remark \ref{remark2}, we consider regular solutions of the regularized problem. Thus $u$
takes its values in $D^{0}_{\alpha}=H^{2}_{\alpha}(0,1)$. Hence, in the case $0\leq \alpha<1$, the claim follows by Lemma \ref{lemma0}.
Moreover, applying Lemma \ref{lemma1}, the same conclusion still holds true if $1<\alpha<2$.

It follows that
\begin{equation}\label{therm3}
\int_{0}^{T}\!\!\!\int_{0}^{1}x^{\alpha-1}\left(\frac{u^2}{2}\right)_{x}\,dx\,dt
=-\frac{(\alpha-1)}{2} \int_{0}^{T}\!\!\!\int_{0}^{1} \frac{u^{2}(t, x)}{x^{2-\alpha}}\,dx\,dt.
\end{equation}
Then the identity \eqref{egaconti} follows by inserting \eqref{therm1} and \eqref{therm3} into \eqref{prbxux}.

Next, we proceed to prove \eqref{directineq}. At this step, we need to distinguish two cases.

\textbf{If $\alpha \in[0,1)$.} For any $\mu\leq \mu(\alpha)$, we have 
\begin{equation}\label{stimaa}
\begin{aligned}
\left|\int_{0}^{1} x u_{x}(t, x) u_{t}(t, x)\,dx\right| & \leq \frac{1}{2} \int_{0}^{1}\left\{u_{t}^{2}(t, x)+x^{2} u_{x}^{2}(t, x)\right\}\,dx \\
&(\text{by \eqref{keyineq1}})\\
&\leq \frac{1}{2} \int_{0}^{1}\left\{u_{t}^{2}(t, x)+x^{\alpha} u_{x}^{2}(t, x) - \frac{\mu}{x^{2-\alpha}}u^2(t,x)\right\}\,dx\\
&= E_{u}(0) \qquad \forall t \geq 0.
\end{aligned}
\end{equation}

\textbf{If $\alpha \in(1,2)$.} Similarly, by Young inequality and using  \eqref{keyineq2},  we also have for any $\mu\leq \mu(\alpha)$
\begin{equation}\label{stimaastar}
\begin{aligned}
\left|\int_{0}^{1} x u_{x}(t, x) u_{t}(t, x)\,dx\right|
&\leq \frac{C^{\prime}_{\alpha}}{2} \int_{0}^{1}\left\{u_{t}^{2}(t, x)+x^{\alpha} u_{x}^{2}(t, x) - \frac{\mu}{x^{2-\alpha}}u^2(t,x)\right\}\,dx\\
&=C^{\prime}_{\alpha} E_{u}(0) \qquad \forall t \geq 0.
\end{aligned}
\end{equation}
Now, we deduce \eqref{directineq} from \eqref{egaconti}, \eqref{stimaa}, \eqref{stimaastar} and the constancy of the energy. The conclusion has thus been proved for classical solutions.

In order to extend \eqref{directineq} and \eqref{egaconti} to the mild solution associated with the initial data
$\left(u_{0}, u_{1}\right) \in H_{\alpha,0}^{1,\mu}(0,1) \times L^{2}(0,1)$,
it suffices to approximate such data by
$\left(u_{0}^{n}, u_{1}^{n}\right) \in D_{\alpha}^{\mu}\times H_{\alpha,0}^{1,\mu}(0,1)$
and use \eqref{directineq} to show that the normal derivatives
of the corresponding classical solutions give a Cauchy sequence in  $L^{2}(0,T)$.
\end{proof}

%%%%%%%%%%%%%%%%%%%%%%%%%%%%%%%%%%%%%%%%%%%%%%%%%%%%%%%%%%%%%%%%%%%%%%%%%%
%%%%%%%%%%%%%%%%%%%%%%%%%%%%%%%%%%%%%%%%%%%%%%%%%%%%%%%%%%%%%%%%%%%%%%%%
\subsection{Solutions defined by transposition}
Since we are dealing with boundary controls, we need to define the solution of \eqref{boundarycontrolproblem} by the transposition
method in the spirit of \cite{Lions1988}.

\begin{definition}
Let $f \in L^{2}(0,T)$ and let $\left(y_{0}, y_{1}\right) \in L^{2}(0,1) \times H_{\alpha,0}^{-1,\mu}(0,1)$ be
fixed arbitrarily. We say that $y$ is a solution by transposition of \eqref{boundarycontrolproblem} if
$$
y \in C^{1}([0,T]; H_{\alpha,0}^{-1,\mu}(0,1)) \cap C([0,T]; L^{2}(0,1))
$$
satisfies, for all $T>0$ and all, $\left(w_{T}^{0}, w_{T}^{1}\right) \in H_{\alpha,0}^{1,\mu}(0,1) \times L^{2}(0,1)$,
\begin{equation}
\label{identityoftransposition}
\begin{aligned}
\left\langle y^{\prime}(T), w_{T}^{0}\right\rangle_{H_{\alpha,0}^{-1,\mu}(0,1), H_{\alpha,0}^{1,\mu}(0,1)}&-\int_{0}^{1} y (T) w_{T}^{1}\,dx=
\left\langle y_{1}, w(0)\right\rangle_{H_{\alpha,0}^{-1,\mu}(0,1), H_{\alpha,0}^{1,\mu}(0,1)} \\
&-\int_{0}^{1} y_{0} w^{\prime}(0)\,dx-\int_{0}^{T} f(t) w_{x}(t, 1)\,dt.
\end{aligned}
\end{equation}
where $w$ is the solution of the backward equation
\begin{equation}\label{backwardsystem}
\begin{cases}
w_{tt} - (x^\alpha w_x)_x  - \frac{\mu}{x^{2-\alpha}}w =0, & (t,x) \in Q,\\
w(t,1)=0,  & t \in (0,T), \\
\left\{
\begin{array}{ll}
w(t, 0) = 0,\quad & \text{if} \;0 \leq \alpha <1,\\
(x^{\alpha} w_{x})(t, 0)= 0, \quad & \text{if} \; 1 < \alpha <2,
\end{array}
\right.
& t \in (0,T),
\\
w(T,x)=w_T^0(x),\quad w_t(T,x)=w_T^1(x), & x \in (0,1).
\end{cases}
\end{equation}
\end{definition}
Note that equation \eqref{backwardsystem} can be reduced to \eqref{adjointproblem} by changing $t$ in $T-t$. In fact, thanks to the change of variable $u(t, x)=w(T-t, x)$
and according to Theorem \ref{wellposed}, the backward system \eqref{backwardsystem} admits a unique solution
$w \in C^{1}([0,T]; L^{2}(0,1)) \cap C([0,T]; H_{\alpha,0}^{1,\mu}(0,1))$ for each $T > 0$.
Moreover, this solution depends continuously on $\left(w_{T}^{0}, w_{T}^{1}\right) \in H_{\alpha,0}^{1,\mu}(0,1) \times L^{2}(0,1)$
and the energy  $E_{w}$ of $w$ is conserved through time.

Now let us define $\mathcal{L}$ by
$$
\mathcal{L}\left(w_{T}^{0}, w_{T}^{1}\right)=\left\langle y_{1}, w(0)\right\rangle_{H_{\alpha,0}^{-1,\mu}(0,1), H_{\alpha,0}^{1,\mu}(0,1)}
-\int_{0}^{1} y_{0} w^{\prime}(0)\,dx-\int_{0}^{T} f(t) w_{x}(t, 1)\,dt.
$$
In view of the direct inequality \eqref{directineq}, there exists a constant $D > 0$ such that
$$
\int_{0}^{T} w_{x}^{2}(t, 1)\,dt \leq D E_{w}(0)=D E_{w}(T),
$$
where 
$$
E_{w}(T)=\frac{1}{2}\left[\|w_T^1\|_{L^2}^2 +\|w_T^0\|_{H^{1,\mu}_{\alpha,0}}^2 \right].
$$
Thus, the mapping $\mathcal{L}$ is a continuous linear form with respect to $\left(w_{T}^{0}, w_{T}^{1}\right) \in$ $H_{\alpha,0}^{1,\mu}(0,1) \times L^{2}( 0,1)$.
Therefore, from Riesz Theorem, there exist a unique couple $\left(y^T_{0}, y^T_{1}\right) \in L^{2}(0,1) \times H^{-1,\mu}_{\alpha,0}(0,1)$ such that
$$
\forall\left(w_{T}^{0}, w_{T}^{1}\right),\quad \mathcal{L}\left(w_{T}^{0}, w_{T}^{1}\right)=\left\langle y^T_{1}, w_{T}^{0}\right\rangle_{H_{\alpha,0}^{-1,\mu}(0,1), H_{\alpha,0}^{1,\mu}(0,1)}-\int_{0}^{1} y^T_{0} w_{T}^{1}\,dx.
$$
Hence, there is a unique solution by transposition
$y \in C^{1}([0,T]; H_{\alpha,0}^{-1,\mu}( 0,1)) \cap C([0,T]; L^{2}( 0,1))$
of \eqref{boundarycontrolproblem} (we refer to \cite[Theorem 4.2, page 46–54]{Lions1988} for the details).

%%%%%%%%%%%%%%%%%%%%%%%%%%%%%%%%%%%%%%%%%%%%%%%%%%%%%%%%%%%%%%%%%%%%%%%%%%%%%%%%%%%%%%%%%%%%%%%%%%%%%%%%%%%%%%%%%%%%%%%%%%%%%%%%%%%%%%%%%
%%%%%%%%%%%%%%%%%%%%%%%%%%%%%%%%%%%%%%%%%%%%%%%%%%%%%%%%%%%%%%%%%%%%%%%%%%%%%%%%%%%%%%%%%%%%%%%%%%%%%%%%%%%%%%%%%%%%%%%%%%%%%%%%
\section{Boundary observability}\label{section3}
In this section the problem of boundary observability of the degenerate/singular wave equation \eqref{adjointproblem} is studied.
Our result guarantees the observability of system \eqref{adjointproblem} under the condition \eqref{basicass}. In order to prove such a result, we first prove the following identity.
\begin{lemma}
For any mild solution $u$ of \eqref{adjointproblem} we have that, for each $T> 0$,
\begin{equation}
\label{lemega}
\int_{0}^{T}\!\!\!\int_{0}^{1}\left\{x^\alpha u_{x}^{2}(t, x)-\frac{\mu}{x^{2-\alpha}}u^2(t,x)-u_{t}^{2}(t, x)\right\}\,dt\,dx+\left[\int_{0}^{1} u(t, x) u_{t}(t, x)\,dx\right]_{t=0}^{t=T}=0.
\end{equation}
\end{lemma}
\begin{proof}
Suppose that $u$ is a classical solution of  \eqref{adjointproblem}. Multiplying the wave equation \eqref{adjointproblem} by $u$ and integrating over
$(0,1)\times(0,T)$, we obtain
\begin{equation*}
\begin{aligned}
0 &=\int_{0}^{T}\!\!\!\int_{0}^{1} u(t, x)\left(u_{t t}(t, x)-\left(x^\alpha u_{x}(t, x)\right)_{x}-\frac{\mu}{x^{2-\alpha}}u(t,x)\right)\,dx\,dt \\
&=\left[\int_{0}^{1} u(t, x) u_{t}(t, x)\,dx\right]_{t=0}^{t=T}-\int_{0}^{T}\!\!\!\int_{0}^{1} u_{t}^{2}(t, x)\,dx\,dt \\
&-\int_{0}^{T}\left[x^\alpha u(t, x) u_{x}(t, x)\right]_{x=0}^{x=1}\,dt+\int_{0}^{T}\!\!\!\int_{0}^{1}x^\alpha u_{x}^{2}(t, x)\,dx\,dt-\int_{0}^{T}\!\!\!\int_{0}^{1}\frac{\mu}{x^{2-\alpha}}u^2(t,x)\,dx\,dt.
\end{aligned}
\end{equation*}
The conclusion follows from the above identity because $x^\alpha u(t, x) u_{x}(t, x)$ vanishes
at $x=0, 1$, owing to \cite[Proposition 2.5]{Alabau2016} and also Remark \ref{remark2}. An approximation argument allows usto extend the conclusion to mild solutions.
\end{proof}

We are now ready to prove the following inverse or observability inequality.
\begin{theorem}\label{observabilityresult}
Assume \eqref{basicass} holds and let $u$ be the mild solution of \eqref{adjointproblem}.
Then, for every $T> 0$,
\begin{equation}
\label{obineq}
\int_{0}^{T} u_{x}^{2}(t, 1)\,dt \geq \left\{\left(2-\alpha\right) T-4\right\} E_{u}(0).
\end{equation}
\end{theorem}
\begin{proof} As usual, we suppose that $u$ is a classical solution of \eqref{adjointproblem}, since the case of mild solutions can be recovered by approximation arguments.
By adding to the right-hand side of  \eqref{egaconti} the left side of \eqref{lemega} multiplied by $\frac{\alpha}{2}$, we obtain
\begin{equation*}
\begin{aligned}
\int_{0}^T u_x^2(t,1)\,dt&=(1-\frac{\alpha}{2})\int_{0}^{T}\!\!\!\int_{0}^{1}\left\{u_{t}^{2}(t, x)+x^\alpha u_{x}^{2}(t,x)-\frac{\mu}{x^{2-\alpha}}u^2(t,x)\right\}\,dx\,dt \\
&\quad +2\left[\int_{0}^{1} x u_{x}(t, x) u_{t}(t, x)\,dx\right]_{t=0}^{t=T} +\frac{\alpha}{2}\left[\int_{0}^{1} u(t, x) u_{t}(t, x)\,dx\right]_{t=0}^{t=T}.
\end{aligned}
\end{equation*} 
Using the definition of $E_u$ and having in mind \eqref{constenergy}, this can be rewritten as
\begin{equation}\label{stimaa0}
\frac{1}{2}\int_{0}^T u_x^2(t,1)\,dt=(1-\frac{\alpha}{2})T E_u(0)
+\left[\int_{0}^{1} u_{t}(t, x)\left(x u_{x}(t, x) +\frac{\alpha}{4}u(t,x)\right)\,dx\right]_{t=0}^{t=T}.
\end{equation} 
We now estimate the last term on the right-hand side of the inequality \eqref{stimaa0}.
%The following estimates are valid for both $t = 0$ and $t =T$.
First we write
\begin{equation}\label{stma0}
\left|\int_{0}^{1} u_{t}(t, x)\left(x u_{x}(t, x) +\frac{\alpha}{4}u(t,x)\right)\,dx\right|  \leq \frac{1}{2} \|u_t\|^{2}_{L^2(0,1)} + \frac{1}{2}\|x u_{x} +\frac{\alpha}{4}u\|^{2}_{L^2(0,1)}.
\end{equation}
Next, we compute:
\begin{equation*}
\begin{aligned}
\|x u_{x} +\frac{\alpha}{4}u\|^{2}_{L^2(0,1)} &= \int_{0}^{1}x^2 u_x^2\,dx + \frac{\alpha}{4}\int_{0}^{1} x (u^2)_x\,dx + \frac{\alpha^2}{16}\int_{0}^{1} u^2\,dx\\
&=\int_{0}^{1}x^2 u_x^2\,dx + \left(\frac{\alpha^2}{16}- \frac{\alpha}{4}\right)\int_{0}^{1} u^2\,dx.
\end{aligned}
\end{equation*}
Using \eqref{keyineq}, it follows that
\begin{equation*}
\begin{aligned}
\|x u_{x} +\frac{\alpha}{4}u\|^{2}_{L^2(0,1)} & \leq \|u\|^{2}_{H^{1,\mu(\alpha)}_{\alpha,0}} +\left[\frac{(1-\alpha)(\alpha-3)}{4} + \frac{\alpha^2-4\alpha}{16}\right]\int_{0}^{1} u^2\,dx \\
&\leq \|u\|^{2}_{H^{1,\mu(\alpha)}_{\alpha,0}}.
\end{aligned}
\end{equation*}
Hence \eqref{stma0} becomes
\begin{equation*}
\left|\int_{0}^{1} u_{t}(t, x)\left(x u_{x}(t, x) +\frac{\alpha}{4}u(t,x)\right)\,dx\right|  \leq \frac{1}{2} \|u_t\|^{2}_{L^2(0,1)} + \frac{1}{2}\|u\|^{2}_{H^{1,\mu(\alpha)}_{\alpha,0}}.
\end{equation*}
Since $\|\cdot\|_{H^{1,\mu(\alpha)}_{\alpha,0}}\leq \|\cdot\|_{H^{1,\mu}_{\alpha,0}}$ ($\forall \mu \leq \mu(\alpha)$), we get
\begin{equation*}
\left|\int_{0}^{1} u_{t}(t, x)\left(x u_{x}(t, x) +\frac{\alpha}{4}u(t,x)\right)\,dx\right| \leq E_u(t)=E_u(0).
\end{equation*}
Combining this last estimate together with \eqref{stimaa0}, we obtain the observability inequality \eqref{obineq} with explicit constants.
\end{proof}

We recall that \eqref{adjointproblem} is said to be observable in time $T>0$ if there exists a constant $C>0$ such that for every $\left(u_{0}, u_{1}\right) \in$ $H_{\alpha,0}^{1,\mu}(0,1) \times L^{2}(0,1)$, the mild solution of \eqref{adjointproblem} satisfies
$$
\int_{0}^{T} u_{x}^{2}(t,1)\,dt \geq C E_{u}(0).
$$
To make sure that the constant $C>0$, we have to impose
$$
T>T_{\alpha}:=\frac{4}{2-\alpha}.
$$
Summarizing, the main observability result is as follows:
\begin{corollary}
Assume \eqref{basicass}. Then \eqref{adjointproblem} is observable in time $T$, provided that 
\begin{equation}\label{timeofcontrol}
T>T_{\alpha}.
\end{equation}
\end{corollary}
\begin{remark}\label{remark3}
Notice that when  $\alpha= 0$, the system \eqref{boundarycontrolproblem} is a non-degenerate linear wave
equation perturbed by a singular inverse-square potential.
By the known controllability results in \cite{VanZua2009} and \cite{Cazacu2012}, the observability time is $T_0 = 2$, which coincides with the classical observability time
for the classical wave equation, see \cite{BLR1992}. Letting $\alpha$ tend to zero, one can find that $\displaystyle\lim_{\alpha \rightarrow 0} T_\alpha= T_0$ and thus the above result complements those in \cite{VanZua2009} and \cite{Cazacu2012} in the $1$-dimensional
case.
\end{remark}
%%%%%%%%%%%%%%%%%%%%%%%%%%%%%%%%%%%%%%%%%%%%%%%%%%%%%%%%%%%%%%%%%%%%%%%%%%%%%%%%%%%%%%%%%%%%%%%%%%%%%%%%%%%%%%%%%%%%%%%%%%%%%%%%%%%%%%%%
%%%%%%%%%%%%%%%%%%%%%%%%%%%%%%%%%%%%%%%%%%%%%%%%%%%%%%%%%%%%%%%%%%%%%%%%%%%%%%%%%%%%%%%%%%%%%%%%%%%%%%%%%%%%%%%%%%%%%%%%%%%%%%%%%
\section{Boundary controllability}
In this section, we study the boundary controllability of the degenerate and singular wave equation \eqref{boundarycontrolproblem}.
Our main result guarantees the exact controllability of system \eqref{boundarycontrolproblem} under the condition \eqref{basicass}. More precisely, the main controllability result of this paper can be stated as follows.
\begin{theorem}\label{controllabilityresult}
Assume \eqref{basicass} holds. Then, for every  $T >T_{\alpha}$ and for any
$\left(y_{0}, y_{1}\right) \in L^{2}(0,1) \times H_{\alpha,0}^{-1,\mu}(0,1)$, there exists a control $f\in \mathrm{L}^2(0, T)$ such that the
solution of \eqref{boundarycontrolproblem} (in the sense of transposition) satisfies $\left(y, y_{t}\right)(T,.)=(0, 0)$.
\end{theorem}

\begin{proof} The proof relies on the use of the Hilbert uniqueness method (HUM) introduced by J.-L. Lions in \cite{Lions1988}. Let $\left(y_{0}, y_{1}\right) \in L^{2}(0,1) \times H_{\alpha,0}^{-1,\mu}(0,1)$, $(w_T^0 , w_T^1 ), (v_T^0 , v_T^1 ) \in H_{\alpha,0}^{1,\mu}(0,1) \times L^{2}(0,1)$ be arbitrary pairs. Let $w$ and $v$ be the mild solutions of the backward problem \eqref{backwardsystem} with final conditions $W^T:=(w_T^0 , w_T^1 )$ and $V^T:=(v_T^0 , v_T^1 )$, respectively.
Let us consider the bilinear form $\Lambda$ on $H_{\alpha,0}^{1,\mu}(0,1) \times L^{2}(0,1)$ defined by
$$
\Lambda\left(W^{T}, V^{T}\right):= \int_{0}^{T} w_{x}(t, 1) v_{x}(t, 1)\,dt \quad \forall W^{T}, V^{T} \in H_{\alpha,0}^{1,\mu}(0,1) \times L^{2}(0,1).
$$
From the direct inequality \eqref{directineq}, it is clear that $\Lambda$ is continuous on $H_{\alpha,0}^{1,\mu}(0,1) \times L^{2}(0,1)$.
On the other hand, thanks to the observability inequality \eqref{obineq}, $\Lambda$ is coercive on $H_{\alpha,0}^{1,\mu}(0,1) \times L^{2}(0,1)$
provided $T > T_\alpha$.

Next, we define the continuous linear map
$$
\ell\left(V^{T}\right):=\left\langle y_{1}, v(0)\right\rangle_{H_{\alpha,0}^{-1,\mu}(0,1), H_{\alpha,0}^{1,\mu}(0,1)}-\int_{0}^{1} y_{0} v^{\prime}(0)\,dx, \quad \forall V^{T} \in H_{\alpha,0}^{1,\mu}(0,1) \times L^{2}(0,1).
$$
Since $\Lambda$ is continuous and coercive on $H_{\alpha,0}^{1,\mu}(0,1) \times L^{2}(0,1)$, and $\ell$ is continuous on
the Hilbert space $H_{\alpha,0}^{1,\mu}(0,1) \times L^{2}(0,1)$, by the
Lax-Milgram Lemma, the variational problem
$$
\Lambda\left(W^{T}, V^{T}\right)=\ell\left(V^{T}\right) \quad \forall V^{T} \in H_{\alpha,0}^{1,\mu}(0,1) \times L^{2}(0,1)
$$
has a unique solution $W_{T} \in H_{\alpha,0}^{1,\mu}(0,1) \times L^{2}(0,1)$.
Then setting $f=w_{x}(t, 1)$ and $T> T_\alpha$, where $w \in C([0,T]; H_{\alpha,0}^{1,\mu}(0,1)) \cap C^1([0,T];L^{2}(0,1))$ is the mild solution of the
backward problem \eqref{backwardsystem} with $(w_T^0 , w_T^1)$ as the final data, we see that
\begin{equation}\label{key1}
\begin{aligned}
&\int_{0}^{T} f(t) v_{x}(t, 1)\,dt= \int_{0}^{T} w_{x}(t, 1) v_{x}(t, 1)\,dt= \Lambda \left(W^{T}, V^{T}\right)  \\
&\quad =\left\langle y_{1}, v(0)\right\rangle_{H_{\alpha,0}^{-1,\mu}(0,1), H_{\alpha,0}^{1,\mu}(0,1)}-\int_{0}^{1} y_{0} v^{\prime}(0)\,dx, \quad\forall  V^{T}\in H_{\alpha,0}^{1,\mu}(0,1) \times L^{2}(0,1).
\end{aligned}
\end{equation}
On the other hand, if $y$ is the solution by transposition of the problem \eqref{boundarycontrolproblem}, then
equality \eqref{identityoftransposition} implies that, for all $V^{T}\in H_{\alpha,0}^{1,\mu}(0,1) \times L^{2}(0,1)$, we have
\begin{equation}\label{key2}
\begin{aligned}
-\int_{0}^{T} f(t) v_{x}(t, 1)\,dt &=\left\langle y^{\prime}(T), v_{T}^{0}\right\rangle_{H_{\alpha,0}^{-1,\mu}(0,1), H_{\alpha,0}^{1,\mu}(0,1)}-\int_{0}^{1} y(T) v_{T}^{1}\,dx \\
&\quad -\left\langle y_{1}, v(0)\right\rangle_{H_{\alpha,0}^{-1,\mu}(0,1), H_{\alpha,0}^{1,\mu}(0,1)}+\int_{0}^{1} y_{0} v^{\prime}(0)\,dx.
\end{aligned}
\end{equation}
Comparing the last relations \eqref{key1}-\eqref{key2}, it follows that
$$
\left\langle y^{\prime}(T), v_{T}^{0}\right\rangle_{H_{\alpha,0}^{-1,\mu}(0,1), H_{\alpha,0}^{1,\mu}(0,1)}-\int_{0}^{1} y(T) v_{T}^{1}\,dx=0, \quad \forall\left(v_{T}^{0}, v_{T}^{1}\right) \in H_{\alpha,0}^{1,\mu}(0,1) \times L^{2}(0,1).
$$
From this we finally deduce that $\left(y(T), y^{\prime}(T)\right)\equiv(0,0)$, i.e. the system \eqref{boundarycontrolproblem} is boundary null controllable
in time $T > T_\alpha$. 
\end{proof}
%%%%%%%%%%%%%%%%%%%%%%%%%%%%%%%%%%%%%%
%%%%%%%%%%%%%%%%%%%%%%%%%%%%%%%%%%%%%%%%
\section{Conclusions and open problems}
In this paper, we have analyzed the boundary controllability of the $1$-D degenerate/singular wave equation. By means of the multiplier method, we have shown that the equation is
observable. As a consequence, applying the Hilbert uniqueness method, we deduced the exact controllability result when the control
acts on the nondegenerate/nonsingular boundary. Moreover, the optimal time of controllability was given.

We present hereafter a non-exhaustive list of comments and open problems related to our work.

\begin{enumerate}
\item As a first thing, we point out that combining our proofs with the ideas of \cite{Van2011}, boundary exact controllability result
can be obtained for the case of a degenerate/singular operator with $\frac{\mu}{x^{\beta}}$ with $\beta \leq 2-\alpha$ instead of $\frac{\mu}{x^{2-\alpha}}$. 
\item In \cite{Alabau2016}, the authors treat the case of a degenerate operator $(a(x)u_x)_x$ with a general coefficient $a(x)$ that vanishes at $x = 0$. It would be interesting to consider a simultaneously degenerate and singular equation with a general degenerate inhomogeneous speed and general singular potential. 
\item Inspired by the results in \cite{KKL}, it would also be interesting to study the wave equation with degeneracy and singularity at the interior of the space domain as done in \cite{FM2016} for the heat equation.
\item Finally, the study of null controllability properties of degenerate or singular coupled wave equations is still to be done and many further directions
remain to be investigated.
\end{enumerate}

%%%%%%%%%%%%%%%%%%%%%%%%%%%%%%%%%%%%%%%%%%%%%%%%%%%%%%%%%%%%%%%%%%%%%%%%%%%%%%%%%%%%%%%%%%%%%%%%%%%%%%%%%%%%%%%%%%%%%%%%%%%%%%%%%%
\section{Appendix}\label{Appendix}
This section is devoted to the proof of Theorem \ref{keytool} which is inspired by \cite[Theorem 1.1]{VanZua2009}.
The main point in the proof is the following change of variables
$$U(x)=x^{\frac{\alpha-1}{2}}u(x),\quad i.e.,\quad u(x)=x^{\frac{1-\alpha}{2}}U(x).$$
Next we compute
\begin{equation*}
\begin{aligned}
\int_{0}^1x^{2} u_x^{2}\,dx &=\int_{0}^1 x^{2}\left(\frac{1-\alpha}{2}x^{\frac{-1-\alpha}{2}}U+x^{\frac{1-\alpha}{2}}U_x\right)^{2}\,dx \\
&=\int_{0}^1\left( \left(\frac{1-\alpha}{2}\right)^2x^{1-\alpha}U^2+x^{3-\alpha}U_x^2 +(1-\alpha)x^{2-\alpha}\left(\frac{U^2}{2}\right)_x\right)\,dx  \\
&=\int_{0}^1\left( \left(\frac{1-\alpha}{2}\right)^2x^{1-\alpha}U^2+x^{3-\alpha}U_x^2 -\frac{(1-\alpha)(2-\alpha)}{2}x^{1-\alpha}U^2\right)\,dx.
\end{aligned}
\end{equation*}
On the other hand, we have 
\begin{equation*}
\begin{aligned}
\int_{0}^1 x^{\alpha} u_x^{2}\,dx &=\int_{0}^1 x^{\alpha}\left(\frac{1-\alpha}{2}x^{\frac{-1-\alpha}{2}}U+x^{\frac{1-\alpha}{2}}U_x\right)^{2}\,dx \\
&=\int_{0}^1\left( \left(\frac{1-\alpha}{2}\right)^2x^{-1}U^2+xU_x^2 +(1-\alpha)\left(\frac{U^2}{2}\right)_x\right)\,dx  \\
&=\int_{0}^1\left( \left(\frac{1-\alpha}{2}\right)^2x^{-1}U^2+xU_x^2\right)\,dx. 
\end{aligned}
\end{equation*}
Therefore,
\begin{equation*}
\begin{aligned}
\int_{0}^1\left( x^{\alpha} u_x^{2}-\mu(\alpha)\frac{u^2}{x^{2-\alpha}}\right)\,dx&=\int_{0}^1\left( \left(\frac{1-\alpha}{2}\right)^2x^{-1}U^2+xU_x^2-\left(\frac{1-\alpha}{2}\right)^2\frac{x^{1-\alpha}}{x^{2-\alpha}}U^2\right)\,dx\\
&=\int_{0}^1 xU_x^2\,dx.
\end{aligned}
\end{equation*}
Hence \eqref{keyineq} may be rewritten exactly as follows:
\begin{equation*}
\int_{0}^1x^{3-\alpha}U_x^2\,dx \leq \int_{0}^1 xU_x^2\,dx.
\end{equation*}
And this inequality is trivially true by the definition of $\alpha$.
%%%%%%%%%%%%%%%%%%%%%%%%%%%%%%%%%%%%%%%%%%%%%%%%%%%%%%%%%%%%%%%%%%%%%%%%%%%%%%%%%%%%%%%%%%%%%%%%%%%%%%%%%%%%%%%%%%%%%%%%%%%%%%%%

%%%%%%%%%%%%%%%%%%%%%%%%%%%%%%%%%%%%%%%%%%%%%%%%%%%%%%%%%%%%%%%%%%%%%%%%%%%%%%%%%%
%=================================


\begin{thebibliography}{99}
\bibitem{Alabau2016}
F. Alabau-Boussouira, P, Cannarsa, G. Leugering,
{\it Control and stabilization of degenerate wave equation}, SIAM J. Control Optim., 55 (2017), 2052-2087.
%==================================

\bibitem{BaiandChai} 
J. Bai, S. Chai,
{\it Exact controllability for some degenerate wave equations},
Mathematical Methods in the Applied Sciences, 43 (2020), 7292-7302.

%======================================
\bibitem{barbu}
V. Barbu,
{\it Partial differential equations and boundary value problems},
Kluwer Academic Publishers, Dordrecht 1998.

%======================================
\bibitem{BLR1992}
C. Bardos, G. Lebeau, J. Rauch,
{\it Sharp sufficient conditions for the observation, control and stabilization of waves from the boundary},
SIAM J. Controland Optimisation, 30 (1992), 1024–1065.

%======================================
\bibitem{Biccari2021}
U. Biccari, V. Hernández-Santamaría, J. Vancostenoble,
{\it Existence and cost of boundary controls for a degenerate/singular parabolic equation},
Mathematical Control \& Related Fields,  doi: 10.3934/mcrf.2021032.

%======================================
\bibitem{BDE}
R. Bosi, J. Dolbeault, and M.J. Esteban,
{\it Estimates for the optimal constants in multipolar Hardy inequalities for Schrodinger and Dirac operators},
Commun. Pure Appl. Anal., 7 (2008), 533-562.

%======================================
\bibitem{Burq1997}
N. Burq, P. Gerard,
{\it Condition n\'cessaire et suffisante pour la contrôlabilit\'e exacte des ondes},
C. R. Acad. Sci. Paris S\'er. I Math., 325 (1997), 749–752.

%%=====================================
\bibitem{CMV2008}
P. Cannarsa, P. Martinez and J. Vancostenoble,
{\it Carleman estimates for a class of degenerate parabolic operators},
SIAM J. Control Optim. 47 (2008), 1-19.
%===========================================
\bibitem{CRV}
P. Cannarsa, D. Rocchetti and J. Vancostenoble,
{\it Generation of analytic semi-groups in $L^2$ for a class of second order degenerate elliptic operators}, Control Cybernet., 37 (2008), 831-878.

%==============================
\bibitem{Cazacu2012}
C. Cazacu,
{\it Schrödinger operators with boundary singularities: Hardy inequality, Pohozaev identity and controllability results},
J. Funct. Anal., 263 (2012), 3741-3783.

%===========================================================
\bibitem{Coron2007}
J.-M. Coron,
{\it Control and nonlinearity}, Mathematical Surveys and Monographs, 136, American
Mathematical Society, Providence, RI, 2007.

%==============================================================

\bibitem{Davies1995}
E. B. Davies, {\it Spectral theory and differential operators}, Cambridge Studies in Advanced
Mathematics, vol. 42, Cambridge University Press, Cambridge, 1995.

%%=====================================
\bibitem{Fardigola2019}
L.V. Fardigola,
{\it Transformation operators in control problems for a degenerate wave equation with variable coefficients},
Ukrainian Math. J. 70 (2019), 1300–1318.
%====================
\bibitem{FMpress} 
G. Fragnelli, D. Mugnai, {\it Control of degenerate and singular parabolic equation}, BCAM Springer Brief, ISBN 978-3-030-69348-0.


%%=====================================
\bibitem{FM2016}
G. Fragnelli, D. Mugnai,
{\it Carleman estimates for singular parabolic equations with interior
degeneracy and non smooth coefficients},
Adv. Nonlinear Anal., 6 (2017), 61-84.


%====================
\bibitem{Gueye2014} 
M. Gueye, {\it Exact boundary controllability of $1$-D parabolic and hyperbolic degenerate equations},
SIAM J. Control Optim., 52 (2014), 2037-2054.


%%=====================================
\bibitem{Ho1986}
Lop Fat Ho,
{\it Observabilit\'e frontiere de l'\'equation des ondes}, C. R. Acad. Sci. Paris S\'er. I Math. 302 (1986), 443-446.

%%=====================================
\bibitem{KKL}
P.I. Kogut, O.P. Kupenko, G. Leugering,
{\it On boundary exact controllability of one-dimensional wave equations with weak and strong interior degeneration}, arXiv:2103.08260.

%%=====================================
\bibitem{Komornik1995}
V. Komornik,
{\it Exact controllability and stabilization (the multiplier method)},
Wiley, Masson, Paris, 1995.

%%=====================================
\bibitem{Lions1988}
J.-L. Lions,
{\it Contrôlabilit\'e exacte, perturbation et stabilisation de Systemes Distribu\'es},
1, Masson, Paris, 1988.

%%=====================================
\bibitem{Lions1988b}
J.-L. Lions,
{\it Exact controllability, stabilization and perturbations for distributed systems},
SIAM Rev., 30 (1988), 1–68.

%%=====================================
\bibitem{Osses2001}
A. Osses,
{\it A rotated multiplier method applied to the controllability of waves, elasticity and tangential Stokes control},
SIAM J. on Control Optim. 40 (2001), 777-800.

%============================================

\bibitem{Russell1967}
David L. Russell,
{\it Nonharmonic Fourier series in the control theory of distributed parameter systems}, 
J. Math. Anal. Appl. 18 (1967), 542–560.

%=======================================================
\bibitem{Van2011} J. Vancostenoble,
{\it Improved Hardy-Poincar\'e inequalities and sharp Carleman estimates for degenerate/singular parabolic problems}, Discrete Contin. Dyn. Syst. Ser. S 4 (2011), 761-790.
%==============================================
\bibitem{VanZua2009} J. Vancostenoble, E. Zuazua,
{\it Hardy Inequalities, Observability, and Control for the Wave and Schrödinger Equations with Singular Potentials}, 
SIAM J. Math. Anal., 41 (4), 1508–1532.
%======================================

\bibitem{VazZua2000} J. L. Vazquez, E. Zuazua,
{\it The Hardy inequality and the asymptotic behaviour of the heat equation with an
inverse-square potential}, J. Funct. Anal. 173, 1 (2000), 103-153.

%%=====================================

\bibitem{Zhang2000}
X. Zhang,
{\it Explicit Observability Inequalities for the Wave Equation with Lower Order Terms by Means of Carleman Inequalities}, 
SIAM Journal on Control and Optimization, 39 (2000), 812-834.

%%=====================================
\bibitem{zhang2017a}
M. Zhang, H. Gao,
{\it Null controllability of some degenerate wave equations}, J Syst Sci Complex  29 (2017), 1-15.

%%=====================================

\bibitem{zhang2017b}
M. Zhang, H. Gao,
{\it Persistent regional null controllability of some degenerate wave equations}, 
Math Methods in the App Sci., 40 (2017), 5821-5830.

%%=====================================

\bibitem{zhang2018}
M. Zhang, H. Gao,
{\it Interior controllability of semi-linear degenerate wave equations}, 
J.Math. Anal. Appl. 457 (2018), 10–22.

%===========================================
\bibitem{Zuazua2006} E. Zuazua,
{\it Controllability of Partial Differential Equations}, 
3rd cycle. Castro Urdiales (Espagne), 2006.
%=======================================================

\bibitem{Zuazua1993} E. Zuazua,
{\it Exact controllability for semilinear wave equations in one space dimension}, 
Inst. H. Poincar\'e Anal. Non Lin\'eaire, 10 (1993), 109–129.
%=======================================================

\end{thebibliography}
\end{document}